\documentclass[preprint,11pt]{elsarticle}

\usepackage{amssymb}
\usepackage{amsthm}
\usepackage{tikz}
\usepackage{tikz,pgfplots}
\usetikzlibrary{decorations.markings}

\usepackage[margin=2.5cm]{geometry}

\journal{Discrete Applied Mathematics}

\begin{document}

	\newtheorem{definition}{Definition}
	\newtheorem{theorem}{Theorem}
	\newtheorem{corollary}{Corollary}
	\newtheorem{lemma}{Lemma}
	\newtheorem{conjecture}{Conjecture}
	
	\tikzset{middlearrow/.style={
			decoration={markings,
				mark= at position 0.7 with {\arrow[scale=2]{#1}} ,
			},
			postaction={decorate}
		}
	}

	\begin{frontmatter}
		
		
		\title{On diregular digraphs with degree two and excess three}
		

		
		\author{James Tuite}
		
		\address{Open University, Walton Hall, Milton Keynes}
		\ead{james.tuite@open.ac.uk}
		
		\begin{abstract}
			Moore digraphs, that is digraphs with out-degree $d$, diameter $k$ and order equal to the
			Moore bound $M(d, k) = 1 + d + d^2 + \dots + d^k$, arise in the study of optimal network
			topologies. In an attempt to find digraphs with a `Moore-like' structure, attention has
			recently been devoted to the study of small digraphs with minimum out-degree $d$ such
			that between any pair of vertices $u,v$ there is at most one directed path of length $\leq k$
			from $u$ to $v$; such a digraph has order $M(d,k)+\epsilon $ for some small excess $\epsilon $. Sillasen et al. have shown that there are no digraphs with minimum out-degree two and excess one
			\cite{MirSil,Sil}. The present author has classified all digraphs with
			out-degree two and excess two \cite{Tui,Tui2}. In this paper it is proven that there
			are no diregular digraphs with out-degree two and excess three for $k \geq 3$, thereby
			providing the first classification of digraphs with order three away from the Moore
			bound for a fixed out-degree.
		\end{abstract}
		
		
	\end{frontmatter}

	\section{Introduction}
	\label{S:1}
	
	The undirected degree/diameter problem asks for the largest possible order of a graph $G$ with given maximum degree
	$d$ and diameter $k$. This problem has applications in the design of efficient networks. A natural upper bound on the order of such a graph is
	$|V(G)| \leq 1 + d + d(d-1) + d(d-1)^2 +\dots + d(d-1)^{k-1}$, where the right-hand side of the inequality is the (undirected) \emph{Moore bound}. A graph is \emph{Moore} if it attains this upper bound. A graph is Moore if and only if it is regular with degree $d$, has diameter $k$ and girth $2k + 1$. The girth condition implies that a Moore graph is $k$-geodetic, i.e. any two vertices are connected by at most one non-backtracking walk of length not exceeding $k$. In the classic paper \cite{HofSin} Hoffman and Singleton show that for diameter $k = 2$ the Moore bound
	is achieved only for degrees $d = 2$, $3$, $7$ and possibly $57$. The unique Moore graphs for $k = 2$ and $d = 2$, $3$ and $7$ are the	$5$-cycle, the Petersen graph and the Hoffman–Singleton graph respectively. The existence of a Moore graph (or graphs)	with diameter $k = 2$ and degree $d = 57$ is a famous open problem. It was later shown by other authors \cite{BanIto,Dam} that for
	diameters $k \geq 3$ Moore graphs exist only in the trivial case $d = 2$.
	
	Given the scarcity of Moore graphs, it is of great interest to find graphs with a `Moore-like' structure. A survey of this problem is given in \cite{MilSir2}. Graphs with maximum degree $d$, diameter $k$ and order $\delta $ less than the Moore bound for some small defect $\delta $ have been studied intensively. In such graphs walks with length $\leq k$ between pairs of vertices are not necessarily unique; associated with each vertex $u$ is a repeat multiset $R(u)$, such that $v \in V(G)$ appears $t$ times in $R(u)$ if
	and only if there are $t+1$ distinct non-backtracking walks with length $\leq k$ between $u$ and $v$. An important result in this direction is that the only graphs with defect one are cycles of length $2k$ \cite{ErdFajHof,BanIto2,KurTsu}.
	
	Alternatively, one can preserve the $k$-geodecity condition and ask for the smallest $d$-regular graphs with girth $2k + 1$.	This is known as the degree/girth problem. A survey of this problem is given in \cite{ExoJaj}. A graph with minimal order subject to the above conditions is called a \emph{cage}.
	
	The directed version of the degree/diameter problem was posed in \cite{BriTou}. The Moore bound for a digraph with maximum out-degree $d$ and diameter $k$ is given by
\[ M(d,k) = 1+d+d^2+\dots +d^k.\]
	Similarly to the undirected case, a digraph is Moore if and only if it is out-regular with degree $d$, has diameter $k$ and is $k$-geodetic, i.e. for any (ordered) pair of vertices $u,v$ there is at most one directed walk from $u$ to $v$ with length $\leq k$. Using	spectral analysis, it was shown in \cite{BriTou} that Moore digraphs exist only in the trivial cases $d = 1$ and $k = 1$, the Moore digraphs being directed cycles of length $k + 1$ and complete digraphs of order $d + 1$ respectively.
	
	There is an extensive literature on digraphs with maximum out-degree $d$, diameter $k$ and order $M(d,k) -\delta $ for small defects $\delta $. Such digraphs arise from removing the $k$-geodecity condition in the requirements for a digraph to be Moore. As	in the undirected case, each vertex $u$ is associated with a repeat multiset $R(u)$, defined in the obvious manner. A digraph with defect $\delta = 1$ is an \emph{almost Moore digraph}; for such a digraph, in place of a set-valued function $R$, we can think of a repeat function $r: V(G) \rightarrow V(G)$. In contrast to the undirected problem, for diameter $k = 2$ there exists an almost Moore
	digraph for every value of $d$ \cite{FioYebAle}. It is known that there are no almost Moore digraphs with $d = 2$ and $k \geq 3$ \cite{MilFri}, $d = 3$ and $k \geq 3$ \cite{BasMilPleZna,BasMilSirSut} or diameters $k = 3, 4$ and $d \geq 2$ \cite{ConGimGonMilMir,ConGimGonMirMor2,ConGimGonMirMor}. It is also shown in \cite{MilSir} that there are no digraphs with degree	$d = 2$ and defect $\delta = 2$ for diameters $k \geq 3$.
	
	Approaching the problem of approximating Moore digraphs from a different perspective, there are several different
	ways to adapt the undirected degree/girth problem to the directed case, as the connection between $k$-geodecity and the	girth does not hold in the directed setting. The directed degree/girth problem, which concerns the minimisation of the	order of out-regular digraphs with given girth, is well developed (see \cite{She} for an introduction). A related problem is	considered in \cite{AraBalOls}. However, the extremal digraphs considered in these problems are in general not $k$-geodetic; in fact, in the directed degree/girth problem, it is conjectured that extremal orders are achieved by circulant digraphs \cite{BehChaWal}.
	
	If we wish to retain the $k$-geodecity condition, but relax the requirement that the diameter should equal $k$, we obtain
	the following problem: \emph{What is the smallest possible order of a $k$-geodetic digraph with minimum out-degree $d$}? A $k$-geodetic digraph $G$ with minimum out-degree $d$ and order $M(d,k) +\epsilon $ is called a $(d,k,+\epsilon )$-digraph, where $\epsilon > 0$ is the \emph{excess} of
	$G$. With each vertex $u$ of a $(d,k,+\epsilon )$-digraph we can associate the set $O(u) = \{ v \in V(G) : d(u, v) \geq k + 1\} $ of vertices	that cannot be reached by $\leq k$-paths from $u$; any element of this set is an \emph{outlier} of $u$. It is known that $(d,k,+1)$-digraphs
	are out-regular with degree $d$ \cite{MirSil}. For a digraph $G$ with excess $\epsilon = 1$, the set-valued function $O$ can be construed as	an \emph{outlier function} $o : V(G) \rightarrow V(G)$, where for each vertex $u$ of $G$ the outlier $o(u)$ of $u$ is the unique vertex of G with
	$d(u,o(u)) \geq k + 1$. We will refer to a $(d,k,+\epsilon )$-digraph with smallest possible excess as a $(d,k)$-geodetic-cage.
	
	The first paper to consider this problem was \cite{Sil}, in which Sillasen proves that there are no diregular $(2,k,+1)$-digraphs for $k \geq 2$. Strong conditions on non-diregular digraphs with excess one were also derived in this paper. These results were later strengthened \cite{MirSil} to show that any digraph with excess $\epsilon = 1$ must be diregular, thereby completing the proof of the non-existence of $(2,k,+1)$-digraphs. It is also known that $(d, k,+1)$-digraphs do not exist for $k = 2$ and $d > 7$ or $k = 3, 4$ for $d > 1$ \cite{MirSil}. In \cite{TuiErs} it is shown that for all $d,k \geq 2$ there exists a diregular $k$-geodetic digraph with out-degree $d$, so that geodetic cages exist for all values of $d$ and $k$, and that for any fixed $k$ the Moore bound can be approached asymptotically by arc-transitive $k$-geodetic digraphs as $d \rightarrow \infty $.
	
	In \cite{Tui} the present author has proven that for $k \geq 2$ any $(2,k,+2)$-digraphs must be diregular. Using an approach similar to that of \cite{MilSir} this analysis was completed in \cite{Tui2} by showing that there are no diregular $(2,k,+2)$-digraphs for $k \geq 3$ and classifying the diregular $(2,2,+2)$-digraphs up to isomorphism. There are exactly two $(2,2,+2)$-digraphs, which are displayed in Fig.\ref{fig:two cages}; these represent the only known non-trivial geodetic cages. New results have allowed the method of \cite{Tui2} to be extended to excess $\epsilon = 3$. In this paper, we therefore present a complete classification of diregular $(2,k,+3)$-digraphs for $k \geq 3$.

	\begin{figure}\centering
		\begin{tikzpicture}[middlearrow=stealth,x=0.2mm,y=-0.2mm,inner sep=0.1mm,scale=1.40,
			thick,vertex/.style={circle,draw,minimum size=5,font=\tiny,fill=black},edge label/.style={fill=white}]
			\tiny
			\node at (50,0) [vertex] (v8) {$$};
			\node at (140,130) [vertex] (v4) {$$};
			\node at (100,30) [vertex] (v7) {$$};
			\node at (-40,130) [vertex] (v5) {$$};
			\node at (100,190) [vertex] (v1) {$$};
			\node at (180,170) [vertex] (v6) {$$};
			\node at (0,30) [vertex] (v3) {$$};
			\node at (0,190) [vertex] (v0) {$$};
			\node at (-80,170) [vertex] (v2) {$$};

			\node at (350,0) [vertex] (w0) {$$};
			\node at (440,130) [vertex] (w1) {$$};
			\node at (400,30) [vertex] (w2) {$$};
			\node at (260,130) [vertex] (w3) {$$};
			\node at (400,190) [vertex] (w4) {$$};
			\node at (480,170) [vertex] (w5) {$$};
			\node at (300,30) [vertex] (w6) {$$};
			\node at (300,190) [vertex] (w7) {$$};
			\node at (220,170) [vertex] (w8) {$$};

			\path
			(w0) edge [middlearrow] (w1)
			(w0) edge [middlearrow] (w2)
			(w1) edge [middlearrow] (w3)
			(w1) edge [middlearrow] (w4)
			(w2) edge [middlearrow] (w5)
			(w2) edge [middlearrow] (w6)
			(w3) edge [middlearrow] (w0)
			(w3) edge [middlearrow] (w8)
			(w4) edge [middlearrow] (w5)
			(w4) edge [middlearrow] (w7)
			(w5) edge [middlearrow] (w1)
			(w5) edge [middlearrow] (w8)
			(w6) edge [middlearrow] (w0)
			(w6) edge [middlearrow] (w4)
			(w7) edge [middlearrow] (w2)
			(w7) edge [middlearrow] (w3)
			(w8) edge [middlearrow] (w6)
			(w8) edge [middlearrow] (w7)
			(v0) edge [middlearrow] (v1)
			(v0) edge [middlearrow] (v2)
			(v1) edge [middlearrow] (v3)
			(v1) edge [middlearrow] (v4)
			(v2) edge [middlearrow] (v5)
			(v2) edge [middlearrow] (v6)
			(v3) edge [middlearrow] (v2)
			(v3) edge [middlearrow] (v7)
			(v4) edge [middlearrow] (v5)
			(v4) edge [middlearrow] (v6)
			(v5) edge [middlearrow] (v0)
			(v5) edge [middlearrow] (v8)
			(v6) edge [middlearrow] (v1)
			(v6) edge [middlearrow] (v7)
			(v7) edge [middlearrow] (v0)
			(v7) edge [middlearrow] (v8)
			(v8) edge [middlearrow] (v3)
			(v8) edge [middlearrow] (v4)

			;
		\end{tikzpicture}
		\caption{The two $(2,2)$-geodetic-cages}
		\label{fig:two cages}
	\end{figure}
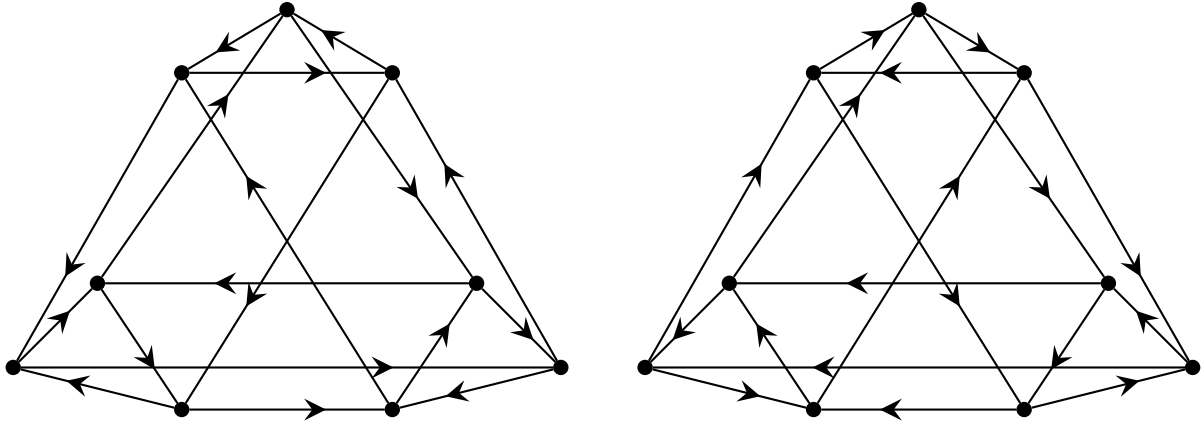

	\section{The neighbourhood lemma}
	
	Let us first establish our notation. Let $G$ be a directed graph with vertex set $V(G)$ and arc set $A(G)$. A directed walk of	length $r$ in $G$ is an alternating sequence of vertices and arcs $v_1e_1v_2e_2 \dots v_re_rv_{r+1}$ such that $e_i$ is an arc from $v_i$ to $v_{i+1}$ for
	$1 \leq i \leq r$. The trivial walk is a walk of length zero that consists of a single vertex. Given vertices $u,v$, the distance $d(u,v)$ from $u$ to $v$ is the length of a shortest walk in $G$ from $u$ to $v$ (we take this to be $\infty $ if there is no such walk); note that we can have $d(u,v) \not = d(v,u)$.	
	
	For vertices $u,v$ we will write $u \rightarrow v$ to indicate that there is an arc from $u$ to $v$ in $A(G)$. The set of out-neighbours of a vertex $u$ of $G$ is $N^+(u) = \{ v \in V(G) : u \rightarrow v\} $; similarly $N^-(u) = \{ v \in V(G) : v \rightarrow u\} $ is the set of in-neighbours of $u$. More generally, for $l > 0$ $N^l(u)$ will stand for the set of vertices that are end-points of walks of length $l$ with initial point $u$ and $N^{-l}(u)$ for the set of vertices that are the initial points of $l$-walks that terminate at $u$. Trivially	$N^0(u) = \{ u\} $, $N^1(u) = N^+(u)$ and $N^{-1}(u) = N^-(u)$. If $S$ is a set of vertices of $G$, then we define $N^+(S)$ to be the multiset $\cup _{v \in S}N^+(v)$. For $0 \leq l \leq k$ the set of vertices that lie within a distance $l$ from a vertex $u$ will be denoted by $T_l(u)$; hence $T_l(u) = \cup _{i=0}^lN^l(u)$. The set $T_{k-1}(u)$ will be written as $T(u)$ for short and will be indicated in diagrams by a triangle based at the vertex $u$. The out-degree of a vertex $u$ is $d^+(u) = |N^+(u)|$ and the in-degree is $d^-(u) = |N^-(u)|$; $G$ is diregular with degree $d$ if $d^+(u) = d^-(u) = d$ for every vertex $u \in V(G)$.
	
In this paper $G$ will stand for a diregular $(d,k,+\epsilon )$-digraph, i.e. a diregular digraph with degree $d$ and order $M(d,k)+\epsilon = 1 + d + d^2+ \dots + d^k + \epsilon $ that is $k$-geodetic, so that for all $u,v \in V(G)$ if there exists a walk $P$ from $u$ to $v$ of length $\leq k$, then it
is the unique such walk. For each vertex $u$ of $G$ there are exactly $\epsilon $ vertices that lie at distance $\geq k+1$ from $u$; the set $O(u)$ of these $\epsilon $ vertices is the \emph{outlier set} of $u$ and each element of $O(u)$ is an \emph{outlier} of $u$. We have $O(u) = V(G) - T_k(u)$.
For $S \subseteq V(G)$ we define $O(S)$ to be the multiset union $\cup {v \in S}O(v)$.

For digraphs with order close to the Moore bound there is a useful interplay between the combinatorial notions of
repeat and outlier and the symmetries of the digraph. For digraphs with defect $\delta = 1$, the repeat function $r$ was shown to be a digraph automorphism in \cite{BasMilPleZna} by a counting argument. This can also be proven by a short matrix argument \cite{Gim}. In her thesis \cite{SilThesis} Sillasen extended this result for almost Moore digraphs to digraphs with larger defects, showing that for any vertex $u$ in a diregular digraph with defect $\delta \geq 2$ the multiset equation $N^+(R(u)) = R(N^+(u))$ holds. This relationship
is known as the \emph{Neighbourhood Lemma}. In \cite{Sil} Sillasen demonstrated that there is a strong analogy between the structure of almost Moore digraphs and digraphs with excess one by proving, by an argument similar to that presented in \cite{Gim}, that the outlier function $o$ of a diregular $(d,k,+1)$-digraph is an automorphism. We now complete this line of reasoning by showing that a Neighbourhood Lemma holds for digraphs with small excess $\epsilon \geq 2$.

	\begin{lemma}[Neighbourhood Lemma]\label{neighbourhoodlemma}
		Let $G$ be a diregular $(d,k,+\epsilon )$-digraph for any $d,k \geq 2$ and $\epsilon \geq 1$. Then for any vertex $u$ of $G$ we have \[ O(N^+(u)) = N^+(O(u))\] as multisets.
	\end{lemma}
	\begin{proof}
	As $G$ is diregular, any vertex can occur at most $d$ times in either multiset. Suppose that a vertex $v$ occurs $t$ times in $N^+(O(u))$. Let $N^-(v) = \{ v_1,v_2, \dots ,v_t ,v_{t+1},\dots ,v_d\} $ and $N^+(u) = \{ u_1,u_2,\dots , u_d\} $, where $O(u) \cap N^-(v) = \{ v_1,v_2,\dots ,v_t \} $. Suppose that $u \not \in N^-(v)$. Since no set $T(u_i)$ contains more than one in-neighbour of $v$ by $k$-geodecity, there are exactly $d-t$ out-neighbours of $u$ that can reach $v$ by a $\leq k$-walk, so that $v$ occurs $t$ times in $O(N^+(u))$. A similar argument deals with the case $u \in N^-(v)$. As both multisets have size $d\epsilon $, this implies the result.
	\end{proof}
	It is pleasing to regard the Neighbourhood Lemma for diregular digraphs with small excess as a limiting case of Lemmas 2 and 3 of \cite{Tui} for non-diregular digraphs.
	
	\section{Main Result}

	For the remainder of this paper $G$ will be a diregular $(2, k,+3)$-digraph for some $k \geq 3$. Geodetic cages for out-degree $d = 2$ and $k = 2$ have been found to have excess two \cite{Tui,Tui2}; for completeness, we mention that there are $(2,2,+3)$-digraphs, both diregular and non-diregular. We will now complete the classification of diregular $(2,k,+3)$-digraphs by showing that for $k \geq 3$ diregular $(2,k,+\epsilon )$-digraphs have excess $\epsilon \geq 4$. The case $k = 3$ is more involved and is discussed
	in the next section.
	
	We employ the following labelling convention for vertices at distance $\leq k$ from a vertex $u$ of $G$. The out-neighbours of $u$ will be labelled according to $N^+(u) = \{ u_1,u_2\} $ and vertices at a greater distance from $u$ are labelled inductively as follows: $N^+(u_1) = \{ u_3,u_4\} $, $N^+(u_2) = \{ u_5,u_6\} $, $N^+(u_3) = \{ u_7,u_8\} $ and so on. See Fig. \ref{fig:kgeq3setup} for an example. A first step in previous studies \cite{MilFri,MilSir,Tui2} of digraphs with degree two and order close to the Moore bound has been to establish the existence of a pair of vertices with exactly one out-neighbour in common. The argument of \cite{Tui2} can be generalised to show that for degree two such a pair exists for any even excess $\epsilon $. For $\epsilon = 3$, we can establish the existence of the necessary pair as
	follows.

	\begin{theorem}\label{nice pair exists}
		For $k \geq 3$, any diregular $(2,k,+3)$-digraph $G$ contains a pair of vertices $u,v$ with exactly one common
		out-neighbour.
	\end{theorem}
	\begin{proof}
		Let $G$ be a diregular $(2,k,+3)$-digraph without the required pair of vertices. Then all out-neighbourhoods are either disjoint or identical. Then by Heuchenne’s condition $G$ is the line digraph of a digraph $H$ with degree two \cite{Heu}. $H$ must be at least $(k-1)$-geodetic. As $2|V(H)| = |V(G)|$, $H$ must be a $(2,k-1,+2)$-digraph. Since the line digraphs of the
		$(2,2)$-geodetic-cages are not $3$-geodetic and there are no $(2,k,+2)$-digraphs for $k \geq 3$ \cite{Tui2}, we have a contradiction.
	\end{proof}
	
	There is no guarantee that distinct vertices do not have identical out-neighbourhoods; witness the geodetic-cage on
	the left of Fig. \ref{fig:two cages}. However, we can say a great deal about the outlier sets of such vertices. The proof of the following lemma is practically identical to that of the corresponding result for $\epsilon = 2$ in \cite{Tui2} and is omitted.
	
	\begin{lemma}\label{lemma 2}
	Let $z,z'$ be vertices of a $(d,k,+\epsilon )$-digraph $H$ for some $\epsilon \geq 1$. If $N^+(z) = N^+(z')$, then there exists a set $X$ of $\epsilon -1$	vertices of $H$ such that $O(z) = \{ z'\} \cup X$, $O(z') = \{ z\} \cup X$.
	\end{lemma}
	
	We now fix an arbitrary pair of vertices $u,v$ of $G$ with a unique out-neighbour in common. We will assume that
	$u_2 = v_2$, so that, following the vertex labelling convention established earlier, we have the situation shown in Fig. \ref{fig:kgeq3setup}. We will also write $N^-(u_1) = \{ u,u^- \} $, $N^-(v_1) = \{ v,v^- \} $, $N^+(u^-) = \{ u_1,u^+\} $ and $N^+(v^-) = \{ v_1, v^+\} $. It is easily seen that $u^- \not = v$ and $v^- \not = u$.

	We can make some immediate deductions concerning the position of the vertices $u,v$ and $u_2$ in the diagram in Fig. \ref{fig:kgeq3setup}.
	\begin{lemma}\label{lemma 3} 
		$v \in N^{k-1}(u_1) \cup O(u)$ and $u \in N^{k-1}(v_1) \cup O(v)$. If $v \in O(u)$, then $u_2 \in O(u_1)$ and if $u \in O(v)$, then $u_2 \in O(v_1)$.
		\end{lemma}
	\begin{proof} 
	$v$ cannot lie in $T(u)$, or the vertex $u_2$ would be repeated in $T_k(u)$. Also, $v \not in T(u_2)$, or else there would be a $\leq k$-cycle	through $v$. Therefore, if
	$v \not \in O(u)$, then $v \in N^{k-1}(u_1)$. Likewise for the other result. If $v \in O(u)$, then neither in-neighbour
	of $u_2$ lies in $T(u_1)$, so that $u_2 \in O(u_1)$. 
\end{proof}
	
	The following lemma is the main tool in our analysis.
	
	\begin{lemma}[Contraction Lemma]
	 Let $w \in T(v_1)$, with $d(v_1,w) = l$. Suppose that $w \in T(u_1)$, with $d(u_1,w) = m$. Then either
	$m \leq l$ or $w \in N^{k-1}(u_1)$. A similar result holds for $w \in T(u_1)$.
\end{lemma}
\begin{proof}
	Let $w$ be as described and suppose that $m >l$. Consider the set $N^{k-m}(w)$. By construction, $N^{k-m}(w) \subseteq N^k(u_1)$, so by $k$-geodecity $N^{k-m}(w) \cap T(u_1) = \emptyset $. At the same time, we have $l + k - m \leq k - 1$, so $N^{k-m}(w) \subseteq T(v_1)$. This implies
	that $N^{k-m}(w) \cap T(v_2) = N^{k-m}(w) p\cap T(u_2) = \emptyset $. As $V(G) = \{ u\} \cup T(u_1) \cup T(u_2) \cup O(u)$, it follows that $N^{k-m}(w) \subseteq \{ u\} \cup O(u)$. Therefore $|N^{k-m}(w)| = 2^{k-m} \leq 4$, so either $m = k - 1$ or $m = k - 2$. Suppose that $m = k - 2$; then $N^2(w) = \{ u\} \cup O(u)$. Neither $v$ nor $v_1$ lies in $N^2(w)$, so that neither $v$ nor $v_1$ lies in $O(u)$. By $k$-geodecity and Lemma \ref{lemma 3}, $v \in N^{k-1}(u_1)$ and	$v_1 \in T(u_1)$, so that $v_1$ appears twice in $T_k(u_1)$. Thus $m = k - 1$.
\end{proof}

	\begin{figure}\centering
		\begin{tikzpicture}[middlearrow=stealth,x=0.2mm,y=-0.2mm,inner sep=0.1mm,scale=1.95,
			thick,vertex/.style={circle,draw,minimum size=13,font=\tiny,fill=white},edge label/.style={fill=white}]
			\tiny
			\node at (150,0) [vertex] (v0) {$u$};
			\node at (90,50) [vertex] (v1) {$u_1$};
			\node at (200,50) [vertex] (v2) {$u_2$};
			\node at (310,50) [vertex] (v3) {$v_1$};
			\node at (110,100) [vertex] (v4) {$u_4$};
			\node at (180,100) [vertex] (v5) {$u_5$};
			\node at (220,100) [vertex] (v6) {$u_6$};
			\node at (250,0) [vertex] (v7) {$v$};
			\node at (70,100) [vertex] (v9) {$u_3$};
			\node at (290,100) [vertex] (v10) {$v_3$};
			\node at (330,100) [vertex] (v11) {$v_4$};
			\node at (60,125) [vertex] (v12) {$u_7$};
			\node at (100,125) [vertex] (v13) {$u_9$};
			\node at (170,125) [vertex] (v14) {$u_{11}$};
			\node at (210,125) [vertex] (v15) {$u_{13}$};
			\node at (280,125) [vertex] (v16) {$v_7$};
			\node at (320,125) [vertex] (v17) {$v_9$};
			\node at (80,125) [vertex] (v18) {$u_8$};
			\node at (120,125) [vertex] (v19) {$u_{10}$};
			\node at (190,125) [vertex] (v20) {$u_{12}$};
			\node at (230,125) [vertex] (v21) {$u_{14}$};
			\node at (300,125) [vertex] (v22) {$v_8$};
			\node at (340,125) [vertex] (v23) {$v_{10}$};
			\node at (50,0) [vertex] (v24) {$u^-$};
			\node at (350,0) [vertex] (v25) {$v^-$};
			\node at (10,50) [vertex] (v26) {$u^+$};
			\node at (390,50) [vertex] (v27) {$v^+$};
			
			\path
			(v0) edge [middlearrow] (v1)
			(v0) edge [middlearrow] (v2)
			(v1) edge [middlearrow] (v9)
			(v1) edge [middlearrow] (v4)
			(v2) edge [middlearrow] (v5)
			(v2) edge [middlearrow] (v6)
			(v7) edge [middlearrow] (v2)
			(v7) edge [middlearrow] (v3)
			(v3) edge [middlearrow] (v10)
			(v3) edge [middlearrow] (v11)
			
			(v9) edge [middlearrow] (v12)
			(v9) edge [middlearrow] (v18)
			(v4) edge [middlearrow] (v13)
			(v4) edge [middlearrow] (v19)
			(v5) edge [middlearrow] (v14)
			(v5) edge [middlearrow] (v20)
			(v6) edge [middlearrow] (v15)
			(v6) edge [middlearrow] (v21)
			(v10) edge [middlearrow] (v16)
			(v10) edge [middlearrow] (v22)
			(v11) edge [middlearrow] (v17)
			(v11) edge [middlearrow] (v23)
			
			(v24) edge [middlearrow] (v1)
			(v25) edge [middlearrow] (v3)
			(v24) edge [middlearrow] (v26)
			(v25) edge [middlearrow] (v27)
			
			;
			\draw (58,130) -- (38,180);
			\draw (122,130) -- (142,180);
			\draw (168,130) -- (148,180);
			\draw (232,130) -- (252,180);
			\draw (278,130) -- (258,180);
			\draw (342,130) -- (362,180);

			\draw (38,180)--(142,180);
			\draw (148,180)--(252,180);
			\draw (258,180)--(362,180);

		\end{tikzpicture}
		\caption{Configuration for $k \geq 3$}
		\label{fig:kgeq3setup}
	\end{figure}
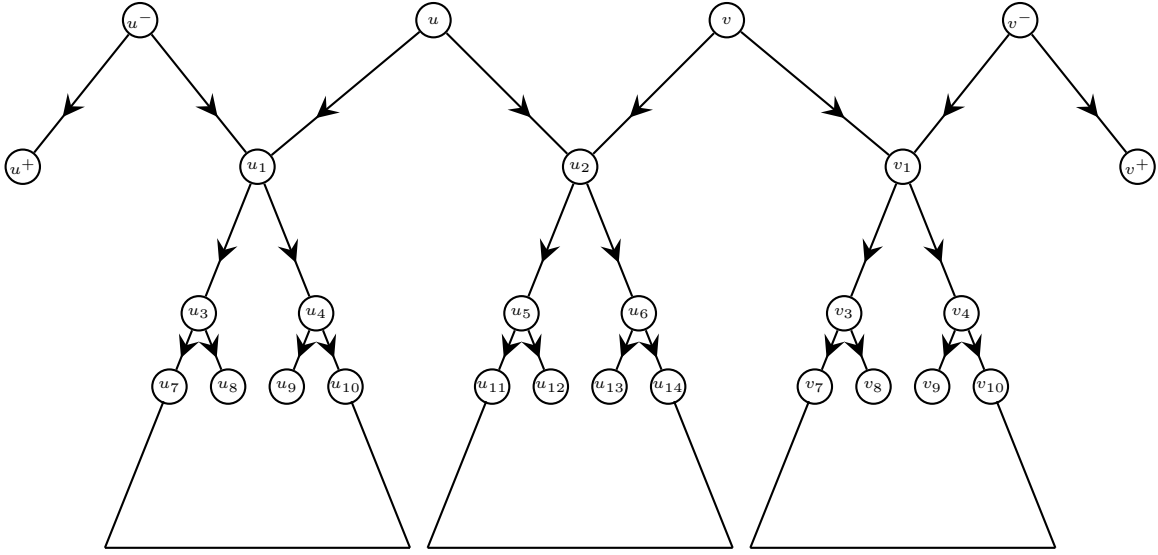
	
	\begin{corollary}\label{corollary 1} 
	If $w \in T(v_1)$, then either $w \in \{ u\} \cup O(u)$ or $w \in T(u_1)$ with $d(u_1,w) = k - 1$ or $d(u_1,w) \leq d(v_1,w)$.
\end{corollary}
	This allows us to restrict the possible positions of $u_1$ and $v_1$ in Fig. \ref{fig:kgeq3setup}.
	
	\begin{corollary}\label{corollary 2}
		$v_1 \in N^{k-1}(u_1) \cup O(u)$ and $u_1 \in N^{k-1}(v_1) \cup O(v)$.
	\end{corollary}
\begin{proof}
	We prove the first inclusion. By Corollary \ref{corollary 1}, $v_1 \in \{ u\} \cup O(u) \cup \{ u_1\} \cup N^{k-1}(u_1)$. By $k$-geodecity $v_1 \not = u$ and by
	construction $v_1 \not = u_1$. 
\end{proof}

	\begin{corollary}\label{corollary 3}
		If $v_1 \not \in O(u)$, then $O(u) = \{ v,v_3,v_4\} $, with a similar result for $v$.
	\end{corollary}
	\begin{proof}
		Suppose that $v_1 \not \in O(u)$; then by the preceding corollary $v_1 \in N^{k-1}(u_1)$. By Lemma \ref{lemma 3} if $v$ is also not an outlier of $u$, then $v \in N^{k-1}(u_1)$, so that $v_1$ appears twice in $T_k(u_1)$. If $v_3 \not \in O(u)$, then by $k$-geodecity $v_3 \in T(u_1)$; however, as $v_1 \rightarrow v_3$, $v_3$ would appear twice in $T_k(u_1)$. The same reasoning applies to $v_4$. 
\end{proof}
	
	\begin{lemma}\label{lemma 5}
		For $k \geq 3$, either $v_1 \in O(u)$ or $u_1 \in O(v)$.
	\end{lemma}
	\begin{proof}
	Suppose that $O(u) = \{ v,v_3,v_4\} $, $O(v) = \{ u,u_3,u_4 \} $. By the Neighbourhood Lemma,	\[ O(\{ u_1,u_2\} ) = O(N^+(u)) = N^+(O(u)) = \{ u_2,v_1,v_7,v_8,v_9,v_{10}\} \]
	and
	\[ O(\{ v_1,u_2\} ) = O(N^+(v)) = N^+(O(v)) = \{ u_2,u_1,u_7,u_8,u_9,u_{10}\} .\]
	By Corollary \ref{corollary 2}, $v_1 \in N^{k-1}(u_1)$ and $u_1 \in N^{k-1}(v_1)$, so we must have $u_1,v_1 \in O(u_2)$. As $O(u_2) \subset N^+(O(u))$, it follows that
	$u_1 \in N^2(v_1)$, so $k = 3$. We can now put $u_9 = v_1$, $v_9 = u_1$. As $N^2(v_1)\cap O(u) = \emptyset $ and $u \not \in N^2(v_1)$, $\{ v_7,v_8,v_{10}\} = \{ u_7,u_8,u_{10}\} $.
	$u_{10} \in N^2(v_1)$ implies that there are two distinct $\leq 3$-paths from $u_4$ to $u_{10}$, contradicting $3$-geodecity.
\end{proof}
	
	We will now identify an outlier of $u$ and $v$ using the Neighbourhood Lemma.
	\begin{theorem}\label{theorem 2}
		For $k \geq 3$, $v_1 \in O(u)$ and $u_1 \in O(v)$.
	\end{theorem}
	\begin{proof}
	Assume for a contradiction that $O(v) = \{ u,u_3,u_4\} $ and $v_1 \in O(u)$. Let $k \geq 4$. $v$ can reach $u_1$ by a $\leq k$-path, so by Corollary \ref{corollary 2} $u_1 \in N^{k-1}(v_1)$. Suppose that $x \in (T_{k-2}(u_1) - \{ u_1\} ) \cap N^{k-1}(v_1)$ and write $N^+(x) = \{ x_1,x_2\} $. Clearly $x_1,x_2 \not \in \{ u,u_3,u_4\} $, so $x_1,x_2 \in T_k(v)$. However, by $k$-geodecity $x_1, x_2 \not \in T(u_2) \cup T(v_1)$, so we are forced to conclude that
	$x_1 = x_2 = v$, which is absurd. It follows from the Contraction Lemma that for any vertex $w \in T_{k-2}(u_1) - \{ u_1,u_3,u_4\} $ we have $d(u_1,w) = d(v_1,w)$. In particular, $N^2(u_1) = N^2(v_1)$. However, as $u_1 \in N^{k-1}(v_1)$, this implies the existence of a
	$(k-1)$-cycle through $u_1$.
	
	Now set $k = 3$. We can put $v_9 = u_1$. $N^2(u_1) \cap O(v) = \emptyset $, so $N^2(u_1) \subset \{ v,v_3,v_4,v_7,v_8, v_{10}\} $. $v_4$ has paths of length $3$ to
	every vertex in $N^2(u_1)$, so  $v_4, v_{10} \not \in N^2(u_1)$, yielding $N^2(u_1) = \{ v,v_3,v_7,v_8\} $. Without loss of generality, $u_7 = v_3$. $u_7 \not \rightarrow u_8$, so $u_8 = v$ and $N^+(v_3) = N^+(u_7) = N^+(u_4)$, which is impossible. 
\end{proof}

	The next stage of our approach is to show that exactly one member of $N^+(v_1)$ is also an outlier of $u$ and similarly for	$v$. This will be accomplished by analysing the possible positions of $u_3$, $u_4$, $v_3$ and $v_4$ in Fig. \ref{fig:kgeq3setup}. The possibilities are described
	in the following lemma.
	\begin{lemma}\label{lemma 6}
		For $k \geq 4$, \[ \{ u_3,u_4\} \subset \{ v_3,v_4\} \cup O(v)\] and \[ \{ v_3,v_4\} \subset \{ u_3,u_4\} \cup O(u).\]
	\end{lemma}
	\begin{proof}
	Let $u_3 \not \in N^+(v1) \cup O(v)$. By Corollary \ref{corollary 1} and Theorem \ref{theorem 2}, $u_3 \in N^{k-1}(v_1)$. By $k$-geodecity, $u_7,u_8 \not \in T(u_2) \cup T(v_1)$.
	
	Also for $k \geq 4$ we cannot have $v \in N^+(u_3)$. Therefore $O(v) = \{ u_1,u_7,u_8\} $. Hence $v$ can reach $u_4$ by a $\leq k$-path. We cannot	have $u_4 \in N^{k-1}(v_1)$, or the same argument would imply that $N^+(u_4) \subset O(v) = \{ u_1,u_7,u_8\} $. By Corollary \ref{corollary 1} we can assume	that $u_4 = v_4$. As $u \not \in O(v)$, we have $u \in N^{k-1}(v_1)$. Since $u_4 = v_4$, to avoid $k$-cycles we must conclude that $u \in N^{k-2}(v_3)$. Likewise $u_3 \in N^{k-2}(v_3$). However, as there is a path $u \rightarrow u_1 \rightarrow u_3$, $v_3$ has a $(k-2)$-path and a $k$-path to $u_3$, which violates
	$k$-geodecity.
\end{proof}
	
	Firstly, we show using the Neighbourhood Lemma that $O(u)$ does not contain both out-neighbours of $v_1$ and vice versa.
	\begin{lemma}\label{lemma 7}
		For $k \geq 4$, $N^+(u_1) \cap N^+(v_1) \not = \emptyset $.
	\end{lemma}
	\begin{proof}
	Suppose that $\{ u_3,u_4\} $ and $\{ v_3,v_4\} $ are disjoint. Then by Theorem \ref{theorem 2} and Lemma \ref{lemma 6} we have $O(u) = \{ v_1,v_3,v_4\} $, $O(v) =
	\{ u_1,u_3,u_4\} $. The Neighbourhood Lemma yields
	\[ N^+(O(v)) = \{ u_3,u_4,u_7,u_8,u_9,u_{10}\} = O(v_1) \cup O(u_2).\]
	Recall that $N^-(u_1) = \{ u^-,u\} $, $N^-(v_1) = \{ v^-, v\} $, $N^+(u^-) = \{ u_1,u^+\} $, $N^+(v^-) = \{ v_1,v^+\} $. Then as $u_2 \not = u^+,v^+$, it follows
	by Theorem \ref{theorem 2} that $u^+ \in O(u)$ and $v^+ \in O(v)$. If $u^+ = v_1$, then, as $T(u_2) \cap (T(u_1) \cup T (v_1)) = \emptyset $, examining $T_k(u^-)$ we
	see that we would have $T(u_2) \subseteq \{ u^-\} \cup O(u^-)$, so that $M(2,k-1) \leq 4$, which is impossible. Without loss of generality,	$u^+ = v_3$, $v^+ = u_3$. Then $v_1$ and $u^-$ have $v_3$ as a unique common out-neighbour, so by Theorem \ref{theorem 2}
	\[ u_1 \in O(v_1) \subset \{ u_3,u_4,u_7,u_8,u_9,u_{10}\} ,\]
	which contradicts $k$-geodecity. 
\end{proof}
	It will now be demonstrated that $u$ cannot reach both out-neighbours of $v_1$ by $\leq k$-paths, so that $O(u)$ contains exactly one out-neighbour of $v_1$, again with a similar result for $v$.
	\begin{lemma}\label{lemma 8}
		For $k \geq 4$, $N^+(u_1) \not = N^+(v_1)$.
		\end{lemma}
	\begin{proof}
	Let $N^+(u_1) = N^+(v_1) = \{ u_3,u_4\} $. If $u$ can reach $v$ by a $\leq k$-path, so that $v \in N^{k-1}(u_1)$, then there would be a $k$-cycle through $v$, so $v \in O(u)$ and $u \in O(v)$. Hence by Lemmas \ref{lemma 2} and \ref{lemma 3}, there exists a vertex $x$ such that $O(u_1) = \{ v_1,u_2,x\} $ and $O(v_1) = \{ u_1,u_2,x\} $. Since $u_1,v_1 \not \in T(u_2)$, $u_3,u_4 \in O(u_2)$. Applying Theorem \ref{theorem 2} to the pairs $(u,u^-)$ and $(u,v)$, we
	see that $u^+,v_1 \in O(u)$. As $N^+(u_1) = N^+(v_1)$, we cannot have $u^+ \in \{ v,v_1\} $. Therefore $O(u) = \{ v,v_1,u^+\} $ and similarly	$O(v) = \{ u,u_1,v^+\} $.
	Suppose that $u^+ = v^+$. Then $u^-$ and $v^-$ have a single common out-neighbour, so that $v_1 \in O(u^-)$ and $u_1 \in O(v^-)$. Hence $u_1 \in O(v) \cap O(v_1) \cap O(v^-)$. As $G$ is diregular, any vertex is the terminal vertex of exactly $M(2,k)$ paths of length $\leq k$, so
	that every vertex is an outlier of exactly three distinct vertices. As $u_2 \not \in \{ v,v_1,v^-\} $, it follows that $u_2$ can reach $u_1$ by a $k$-path; likewise $u_2$ can reach $v_1$. Therefore $u^-, v^- \in N^{k-1}(u_2)$; however, as $u^+ = v^+$, this is impossible. Hence $u^+ \not = v^+$.
	The Neighbourhood Lemma gives
	\[ N^+(O(u)) = \{ v_1,u_2,u_3,u_4\} \cup N^+(u^+) = O(u_1) \cup O(u_2)\]
	and
	\[ N^+(O(v)) = \{ u_1,u_2,u_3,u_4\} \cup N^+(v^+) = O(v_1) \cup O(u_2).\]
	It follows that $O(u_2)$ contains a vertex $z \in N^+(u^+) \cap N^+(v^+)$. Therefore $u^+, v^+ \not \in T(u_2)$. Examining $T_k(u^-)$, we see that $u^+$
	does not lie in $T(u_1)-\{ u_1\} = T(v_1)-\{ v_1\} $. As already mentioned, $u^+ \not = v,v_1$. Therefore $v$ cannot reach $u^+$ by a $\leq k$-path,	so $u^+ \in O(v) = \{ u,u_1,v^+\} $, a contradiction. 
\end{proof}
		
	Since $u,v$ was an arbitrary pair of vertices with a unique common out-neighbour, Lemmas \ref{lemma 6}, \ref{lemma 7} and \ref{lemma 8} imply the following result.
	\begin{corollary}\label{corollary 4} 
	For $k \geq 4$, if $u,v$ are vertices with a single out-neighbour $u_2$ in common, then $v_1 \in O(u)$, $u_1 \in O(v)$ and $|O(u) \cap N^+(v_1)| = |O(v) \cap N^+(u_1)| = 1$.
	\end{corollary}
	Thanks to Corollary \ref{corollary 4} we can assume that $u_3 = v_3$, $u_4 \not = v_4$, $v_1, v_4 \in O(u)$ and $u_1,u_4 \in O(v)$. Repeated applications of
	Corollary \ref{corollary 4} allow us to prove that there are no diregular $(2,k,+3)$-digraphs for $k \geq 4$ by inductively identifying outliers of $u_2$.
	\begin{theorem}\label{theorem 3}
		There are no diregular $(2,k,+3)$-digraphs for $k \geq 4$.
	\end{theorem}
	\begin{proof}
	Let $k \geq 5$. As $u_3 \in N^+(u_1)\cap N^+(v_1)$, $u_3 \in O(u_2)$. The pair $(u_1,v_1)$ has $u_3$ as a unique common out-neighbour, so by Corollary \ref{corollary 4} we can assume that $u_9 = v_9$, $u_{10} \not = v_{10}$. $u_4,v_4,u_9 \not \in T(u_2)$, so $u_9 \in O(u_2)$. The pair $(u_4,v_4)$ has $u_9$ as a unique common out-neighbour, so we can assume that $u_{21} = v_{21}$ and $u_{22} \not = v_{22}$. As $u_{10},v_{10},u_{21} \not \in T(u_2)$, $u_{21} \in O(u_2)$. Continuing	further we see that $u_{45} \in O(u_2)$. In fact, it follows inductively that $O(u_2)$ contains at least $k-1$ distinct vertices, which is
	impossible, as $G$ has excess $\epsilon = 3$.
	
	Now set $k = 4$. By the foregoing reasoning, we can write \[ O(u_2) = \{ u_3,u_9,u_{21}\} , O(u) = \{ v_1,v_4,z\} , O(v) = \{ u_1,u_4,z'\} \]
	for some vertices $z,z'$ and assume that $u_3 = v_3$, $u_9 = v_9$, $u_{21} = v_{21}$ and that $u_{22}$ and $v_{22}$ have a single common out-neighbour.	Trivially $u,v,u_1,v_1,u_4,v_4 \not \in  O(u_2)$. Taking into account adjacencies among $u,v,u_1$ and $v_1$, we can assume
	that $u_{23} \rightarrow u$, $u_{25} \rightarrow v_1$, $u_{27} \rightarrow v$ and $u_{29} \rightarrow u_1$. As $u_1 \rightarrow u_4$, $u_4 \not \in N^3(u_6)$. If $u_4 \in N^2(u_{11})$, then $u_{11}$ has two distinct
	$\leq 4$-paths to $u_4$. Thus $u_4 \in N^2(u_{12})$. However, now there are distinct $\leq 4$-paths from $u_{12}$ to $u_9$, violating $4$-geodecity.
	\end{proof}

	\section{Classification of $(2,3,+3)$-digraphs}\label{section: class 2,3,3 digraphs}
	
	The methods of the preceding section do not settle the issue of the existence of a $(2,3,+3)$-digraph. We now demonstrate how these ideas can be extended to show that there are no such digraphs.   
	
	Let $G$ be a diregular $(2,3,+3)$-digraph and fix an arbitrary pair of vertices $u$ and $v$ with a single out-neighbour $u_2$ in common. Theorem \ref{theorem 2} identifies $v_1$ as an outlier of $u$ and $u_1$ as an outlier of $v$. As in the preceding section, we will analyse the possible positions of $u_3,u_4,v_3$ and $v_4$ in $T_3(u)$ and $T_3(v)$. Our first goal is to show that at least one out-neighbour of $v_1$ is an outlier of $u$ and vice versa. Call a pair of vertices $u,v$ with $|N^+(u) \cap N^+(v)|=1$ \emph{bad} if either $O(u) \cap N^+(v_1)=\emptyset $ or $O(v)\cap N^+(u_1)=\emptyset $. Let $u,v$ be an arbitrary bad pair; without loss of generality $u_3,u_4 \in T(v_1)- \{ v_1\} $.
	
	\begin{lemma}
		$|\{ u_3,u_4\} \cap N^2(v_1)| \leq 1$.	
	\end{lemma}
	\begin{proof}
		Suppose that $u_3,u_4 \in N^2(v_1)$. By 3-geodecity applied to $T_3(v_1)$ and $T_3(u)$, we have $N^2(u_1)\cap T(v_1) = N^2(u_1)\cap T(u_2)=\emptyset $, so that $N^2(u_1)=\{ v\} \cup O(v)$. As $u_1\in O(v)$, there would thus be a 2-cycle through $u_1$.	
	\end{proof}
	
	\begin{corollary}
		If $u,v$ is a bad pair then without loss of generality either $u_3=v_3,u_4=v_4$ or $u_3=v_3,u_4=v_9$.
	\end{corollary}
	
	\begin{lemma}
		If $u,v$ is a bad pair then we can assume that $u_3=v_3,u_4=v_4$.	
	\end{lemma}
	\begin{proof}
		Let $u_3=v_3,u_4=v_9$. $u_1$ and $v_1$ have $u_3$ as unique common out-neighbour, so by Theorem \ref{theorem 2} $u_4 \in O(v_1)$, whereas there is a $2$-path from $v_1$ to $u_4$.
	\end{proof}
	
	\begin{theorem}\label{excess three no bad pairs}
		There are no bad pairs.
	\end{theorem}
	\begin{proof}
		$u,v \not \in T(u_1)-\{ u_1\} = T(v_1)-\{ v_1\}$, so $O(u) = \{ v,v_1,x\} , O(v)= \{ u,u_1,x\} $ for some vertex $x$. By the Neighbourhood Lemma \[ O(u_1)\cup O(u_2) = \{ u_2,v_1,u_3,u_4\} \cup N^+(x), O(v_1) \cup O(u_2) = \{ u_2,u_1,u_3,u_4\} \cup N^+(x).\] As $u_1$ can reach $u_3$ and $u_4$, we have $u_3,u_4 \in O(u_2)$. Applying Lemma \ref{lemma 2} to $u_1$ and $v_1$, we see that $u_1 \in O(v_1), v_1 \in O(u_1), u_3\in O(u_4), u_4\in O(u_3)$. Therefore \[ O(u_1)=\{ u_2,v_1,x_1\} , O(u_2) = \{ u_3,u_4,x_2\} , O(v_1) = \{ u_2,u_1,x_1\} ,\] where $N^{+}(x) = \{ x_1,x_2\} $. Let $N^{-}(u_1) = \{ u,u^{-} \} $ and $N^{-}(v_1) = \{ v,v^{-}\} $.  
		Obviously $|N^{+}(u)\cap N^+(u^{-})| = |N^+(v)\cap N^+(v^{-})| = 1$, so $u_2 \in O(u_1) \cap O(v_1) \cap O(u^{-})\cap O(v^{-})$. $u^{-},v^{-} \not \in \{ u_1,v_1\} $, so it follows that $u^{-} = v^{-}$ and $N^{+}(u^{-}) = \{ u_1,v_1\} $. As $N^{+}(u_1) = N^{+}(v_1)$, this contradicts $3$-geodecity.	
	\end{proof}
	
	Theorem \ref{excess three no bad pairs} shows that we can assume that $v_1,v_4\in O(u)$ and $u_1,u_4\in O(v)$. The next step is to show that $u$ has exactly one outlier in $N^+(v_1)$ and $v$ has exactly one outlier in $N^+(u_1)$.
	\begin{lemma}
		Either $O(u) \not = T_1(v_1)$ or $O(v) \not = T_1(u_1)$.
	\end{lemma}
	\begin{proof}	
		Let $O(u)= \{ v_1,v_3,v_4\} $ and $O(v) = \{ u_1,u_3,u_4\} $. Applying the Neighbourhood Lemma to $u$ and $v$ yields $N^+(O(u)) = \{ v_3,v_4,v_7,v_8,v_9,v_{10} \} = O(u_1)\cup O(u_2)$ and $N^+(O(v)) = \{ u_3,u_4,u_7,u_8,u_9,u_{10}\} = O(v_1)\cup O(u_2)$. If $u_3\in O(u_2)$, then $u_3\in \{ v_3,v_4,v_7,v_8,v_9,v_{10} \} $, contradicting $u_3\in
		O(v)$. Similarly $u_4,v_3,v_4 \not \in O(u_2)$. We can thus assume that \[ O(u_1) = \{ v_3,v_4,v_7\} , O(u_2) = \{ u_8,u_9,u_{10}\} = \{ v_8,v_9,v_{10}\} , O(v_1) = \{ u_3,u_4,u_7\} .\] 
		As $v \not \in O(u)$, $v \in N^2(u_1)$. $\{ u_8,u_9,u_{10}\} = \{ v_8,v_9,v_{10}\} $ implies that $v = u_7$. Likewise $u = v_7$. Therefore $N^3(u_2) = \{ u,u_1,u_3,u_4,u_7=v,v_1,v_3,v_4\} $. $u$ and $v$ have a common out-neighbour, so we can set $u_{11} \rightarrow u , u_{13} \rightarrow v $. We have $u\rightarrow u_1$, so $u_1 \in N^2(u_6)$. If $u_{13} \rightarrow u_1$, then $u_{13}$ would have two $\leq 3$-paths to $v$, so $u_{14}\rightarrow u_1$ and similarly $u_{12}\rightarrow v_1$. $v_1\in N^2(u_5)$ implies that $v_3,v_4\in N^2(u_6)$. $u_{13}$ can already reach $v_3$ and $v_4$ by $3$-paths via $v$, so we are forced to conclude that $N^+(u_{14}) = \{ v_3,v_4\} $, contradicting $u_{14}\rightarrow u_1$.
	\end{proof}
	
	\begin{lemma}
		$O(u) \not = T_1(v_1)$ and $O(v) \not = T_1(u_1)$.
	\end{lemma}
	\begin{proof}
		By the preceding lemma at least one of these inequalities is valid. Let $O(u) = \{ v_1,v_3,v_4\} $ and $u_1,u_3\in O(v)$ but $u_4 \not \in O(v)$. $u_4$ must lie in $N^2(u_1)$, say $u_4 = v_7$. We have $\{ v,v_8,v_9,v_{10}\} \subseteq \{ u,u_7,u_8,u_9,u_{10}\} $. As $u_4 = v_7$, $N^+(u_4) \cap \{ v_8,v_9,v_{10}\} = \emptyset $, so $u_4 \rightarrow v$, say $u_9 = v$, and $\{ v_8,v_9,v_{10}\} = \{ u,u_7,u_8\} $. If $v_8 = u$, then $v_3$ would have two distinct $\leq 3$-paths to $u_4$; hence we can set $u = v_9, u_7 = v_8, u_8 = v_{10}$. As $u_1$ and $v_3$ have unique common out-neighbour $u_4$, from Theorem \ref{theorem 2} it follows that $u_7 \in O(u_1)$, which is plainly false.	
	\end{proof}
	
	\begin{corollary}
		$|O(u) \cap N^+(v_1)| = |O(v) \cap N^+(u_1)| = 1$.	
	\end{corollary}
	
	Without loss of generality, we can assume that $u_3\in T_3(v), v_3\in T_3(u), u_4\in O(v), v_4\in O(u)$. There are now two possibilities, depending upon the distance from $v$ to $u_3$: either $u_3 = v_3$ or we can put $u_3 = v_9, v_3 = u_9$.
	\newline
	\newline
	{\bf Case 1: $u_3 = v_3$} 
	\newline
	\newline
	We can define the vertex $x$ by $O(u) = \{ v_1,v_4,x\} $. Also $u_1,u_4\in O(v)$. Let $N^+(x) = \{ x_1,x_2\} $. 
	As $v,v_9,v_{10} \not \in \{ v_1,v_4\} $, there are three essentially different possibilities: 
	a) $x \not \in \{ v,v_9,v_{10}\} $, b) $x = v$ and c) $x = v_9$.
	\newline
	\newline
	{\bf Case 1.a): $x \not \in \{ v,v_9,v_{10}\} $}
	\newline
	\newline
	In this case $O(u)= \{ v_1,v_4,x\} , O(v) = \{ u_1,u_4,x\} $. $u$ can reach each of $v,v_9,v_{10}$, so $\{ v,v_9,v_{10} \} = \{ u,u_9,u_{10}\} $. We can assume that $u = v_9, v = u_9$ and $u_{10} = v_{10}$. It is obvious that $u_3,u_{10} \in O(u_2)$. $u_1$ and $v_1$ have the single out-neighbour $u_3$ in common, so $v_4,u\in O(u_1), u_4,v\in O(v_1)$. Applying the Neighbourhood Lemma to $u$ and $v$, $N^+(O(u)) = O(u_1)\cup O(u_2) = \{ u_3,v_4,u,u_{10},x_1,x_2\} $, and $N^+(O(v)) = O(v_1) \cup O(u_2) = \{ u_3,u_4,v,u_{10},x_1,x_2\} $, so without loss of generality $O(u_1) = \{ u,v_4,x_1\} , O(u_2) = \{ u_3,u_{10},x_2\} $ and $O(v_1) = \{ u_4,v,x_1\} $. Now, $u^+ \in O(u) = \{ v_1,v_4,x\} $ and by 3-geodecity
	$u^+ \not = v_1,v_4$, so that $u^+ = x$. Similarly $v^+ = x$, so by $3$-geodecity applied to $u^-$ and $v^-$, $x_2\in T(u_2)$, contradicting $x_2\in O(u_2)$.
	\newline
	\newline
	{\bf Case 1.b): $x = v$}
	\newline
	\newline
	Now $O(u) = \{ v,v_1,v_4\} $. $v_9$ and $v_{10}$ are not outliers of $u$, 
	so $N^+(u_4) \cap N^+(v_4) \not = \emptyset $. By Theorem \ref{theorem 2}, $u^+ \in O(u) = \{ v,v_1,v_4\} $, so that $u^-$ has distinct $\leq 3$-paths to either $u_3$ or a vertex in $N^+(u_4)\cap N^+(v_4)$.
	\newline
	\newline
	{\bf Case 1.c): $x = v_9$}
	\newline
	\newline
	In this case $O(u) = \{ v_1,v_4,v_9\} $. As 
	$v_{10} \not \in O(u)$, without loss of generality, either i) $v_{10} = u$ or ii) $v_{10} = u_{10}$. 
	\newline
	\newline
	{\bf Case 1.c)i): $x = v_9, v_{10} = u$}
	\newline
	\newline
	Without loss of generality $v = u_{10}$, so that $O(v) = \{ u_1,u_4,u_9\} $. Evidently $u_3\in O(u_2)$. $u_1$ and $v_1$ have a single common out-neighbour, so $v_4\in O(u_1), u_4\in O(v_1)$ and $|O(u_1)\cap \{ v_9,u\} | = |O(v_1)\cap \{ u_9,v\} | = 1$. Applying the Neighbourhood Lemma, $N^+(O(u)) = O(u_1)\cup O(u_2) = \{ u_3,v_4,v_9,u\} \cup N^+(v_9)$ and $N^+(O(v)) = O(v_1)\cup O(u_2) = \{ u_3,u_4,u_9,v\} \cup N^+(u_9)$. If $u \in O(u_2)$, then $u_9\rightarrow u$ and there are distinct $\leq 3$-paths from $u_4$ to $u_2$, so $u \in O(u_1)$. Similarly $v \in O(v_1)$. Hence $d(u_1,v_9) = d(v_1,u_9) = 3$. If $v_9 \in N^2(u_3)$, then there are two $\leq 3$-paths from $v_1$ to $v_9$, so $u_9\rightarrow v_9$ and, by the same reasoning $v_9\rightarrow u_9$, so that $G$ would contain a digon.
	\newline
	\newline
	{\bf Case 1.c)ii): $x = v_9, v_{10} = u_{10}$}
	\newline
	\newline
	Now $u_9 = v, O(v) = \{ u,u_1,u_4\} $ and $u_3,u_{10} \in O(u_2)$. From $N^+(u_1)\cap N^+(v_1) = \{ u_3\} $, it follows that $v_4,v_9\in O(u_1),u_4,v\in O(v_1)$. The Neighbourhood Lemma yields $N^+(O(u)) = O(u_1)\cup O(u_2) = \{ u_3,v_4,v_9,u_{10}\} \cup N^+(v_9)$ and $N^+(O(v)) = O(v_1)\cup O(u_2) = \{ u_1,u_2,u_3,u_4,v,u_{10}\} $. As
	$u_2 \not \in O(u_2)$, the second equation implies that $O(v_1) = \{ u_4,v,u_2\} , O(u_2) = \{ u_1,u_3,u_{10}\} $ and $O(u_1) = \{ v_4,v_9,y\} $, where $N^+(v_9) = \{ u_1,y\} $. $u^+ \in O(u) = \{ v_1,v_4,v_9\} $. This implies the existence of distinct $\leq 3$-paths from $u^-$ to $u_3,u_{10}$ or $u_1$.
	\newline
	\newline
	{\bf Case 2: $u_3 = v_9, v_3 = u_9$}
	\newline
	\newline
	Write $O(u) = \{ v_1,v_4,x\} , O(v) = \{ u_1,u_4,y\} $. By $3$-geodecity, $v_7,v_8$ and $v_{10}$ do not lie in $\{ u_3,v_3,u_7,u_8\} $, so $\{ v_7,v_8,v_{10}\} = \{ u,u_{10},x\} $. Likewise $\{ u_7,u_8,u_{10}\} = \{ v,v_{10},y\} $.
	$u \not = v_{10}$ and $v \not = u_{10}$, so without loss of generality $u = v_7, v = u_7$ and $\{ v_8,v_{10} \} = \{ u_{10},x\} , \{ u_8,u_{10}\} = \{ v_{10},y\} $. Suppose
	that $u_{10} \not = v_{10}$. Then $u_{10} = v_8 = y$ and $v_{10} = u_8 = x$, so that $u_8\in O(u)$, which is absurd. It follows that $u_{10} = v_{10}$, $v_8 = x$ and $u_8 = y$, so that $O(u)= \{ v_1,v_4,v_8\} , O(v) = \{ u_1,u_4,u_8\} $. By the Neighbourhood Lemma, $N^+(O(u)) = O(u_1)\cup O(u_2) = \{ v_3,v_4,u_3,u_{10}\} \cup N^+(v_8)$ and $N^+(O(v)) = O(v_1)\cup O(u_2) = \{ u_3,u_4,v_3,u_{10}\} \cup N^+(u_8)$. $u_1$ can reach $v_3,u_3$ and $u_{10}$, so $O(u_1) = \{ v_4\} \cup N^+(v_8) , O(u_2) = \{ u_3,v_3,u_{10}\} $ and $O(v_1) = \{ u_4\} \cup N^+(u_8)$. Thus $N^3(u_2) = \{ u,u_1,u_4,u_8,v,v_1,v_4,v_8\} $. We can set $u_{11}\rightarrow u$, $u_{12}\rightarrow v_1$, $u_{13}\rightarrow v$ and $u_{14} \rightarrow u_1$. However, wherever $u_4$ lies in $N^3(u_2$) we have a violation of $3$-geodecity. In conclusion, we have the following theorem.
	
	\begin{theorem}\label{no diregular (2,3,+3) digraphs}
		There are no diregular $(2,3,+3)$-digraphs.	
	\end{theorem}
	
	There exists a Cayley $(2,3,+5)$-digraph with order 20, so $(2,3)$-geodetic cages have order $18$, $19$ or $20$. By
	Theorem \ref{no diregular (2,3,+3) digraphs} if $(2,3)$-cages have order $18$, then they must be non-diregular. It would be of great interest to identify extremal digraphs for these parameters. The question of the existence of non-diregular $(2,k,+3)$-digraphs also remains open.
	\newline
	\newline	
	{\bf Acknowledgements}
	\newline
	\newline
	The author thanks the two anonymous reviewers for their suggestions for improving this article and Prof. J. Sir\a'a\v n and	Dr. G. Erskine for helpful discussion of the manuscript.

\end{document}